\documentclass[10 pt,draft]{amsart}
\usepackage{amssymb}
\usepackage{latexsym}
\usepackage{amsmath}
\usepackage{amsthm}
\usepackage{amsfonts}
\usepackage{amscd}
\usepackage{amsxtra}
\usepackage{xcolor}

\newcommand{\N}{\mathbb{N}}

\newtheorem{theorem}{Theorem}[section]
\newtheorem{lemma}[theorem]{Lemma}

\newtheorem{definition}[theorem]{Definition}

\newtheorem{remark}[theorem]{Remark}
\newtheorem{corollary}[theorem]{Corollary}

\newtheorem{question}[theorem]{Question}

\title[]{Extending Families of Disjoint Hypercyclic Operators}
\author[\"O. Martin]{\"Ozg\"ur Martin}
\address[\"O. Martin]{Department of Mathematics, Mimar Sinan Fine Arts University, Silah\c{s}\"or Cad. 71, Bomonti \c{S}i\c{s}li 34380, Istanbul, Turkey}
\email{ozgur.martin@msgsu.edu.tr}

\author[R. Sanders]{Rebecca Sanders}
\address[R. Sanders]{
	Department of Mathematical and Statistical Sciences, Marquette University, Milwaukee WI 53201}
\email{rebecca.sanders@marquette.edu}

\date{December 5, 2023 }

\begin{document}

\begin{abstract}
    In the present note, we solve two open questions posed by Salas in \cite{Sa4} about disjoint hypercyclic operators. First, we show that given any family $T_1, \dots, T_N$ of disjoint hypercyclic operators, one can always select an operator $T$ such that the extended family $T_1, \dots, T_N, T$ of operators remains disjoint hypercyclic. In fact, we prove that the set of operators $T$ which can extend the family of disjoint hypercyclic operators is dense in the strong operator topology in the algebra of bounded operators. Second, we show the existence of two disjoint weakly mixing operators that fail to possess a dense d-hypercyclic manifold. Thus, these operators satisfy the Disjoint Blow-up/Collapse property but fails to satisfy the Strong Disjoint Blow-up/Collapse property, a notion which was introduced by Salas as a sufficient condition for having a dense linear manifold of disjoint hypercyclic vectors. 
\end{abstract}

\maketitle

\section[Introduction]{Introduction}

Since its introduction by Bernal \cite{BG2} and independently by B\`es and Peris \cite{BeP1}, the notion of disjoint hypercyclicity has become one of the central themes in linear dynamics. Let $X$ denote a separable and infinite dimensional Banach space, and let $\mathcal{B}(X)$ denote the algebra of bounded operators. Recall that an operator $T \in \mathcal{B}(X)$ is called \textit{hypercyclic} if it posseses a dense orbit 
$\mbox{Orb}(T, x) = \{T^nx:n \in \N\}$ for some vector $x \in X$. Such a vector $x$ is called as a \textit{hypercyclic vector} for the operator $T$, and the set of hypercyclic vectors for $T$ is denoted by $\mathcal{HC}(T)$. Disjoint hypercyclicity provides a method for studying the independence of the orbits of a finite family of hypercyclic operators.  

\begin{definition}\label{Disjointness}
    {\rm
	The operators $T_1, \dots, T_N \in \mathcal{B}(X)$, with $N \ge 2$, are \textit{disjoint hypercyclic}, or \textit{d-hypercyclic} for short, if there exists a vector $x \in X$ such that the set $\{ (T_1^n x,\ldots,T_N^n x): n \in \N \}$ is dense in $X^N = X \times \dots \times X$. 
    }
\end{definition}

In other words, the operators $T_1, \dots, T_N$ are disjoint hypercyclic if the direct sum operator $T_1 \oplus \dots \oplus T_N$ has a hypercyclic vector of the form $(x,\ldots,x) \in X^N$. Following the hypercyclicity terminology, such a vector $x$ is also called as a 
\textit{d-hypercyclic vector} for the operators $T_1, \dots, T_N$. The set of d-hypercyclic vectors for the operators $T_1, \dots, T_N$ is denoted by $\mbox{d-}\mathcal{HC} (T_1, \dots, T_N)$.  Moreover, whenever the set $\mbox{d-} \mathcal{HC}(T_1, \dots, T_N)$ of d-hypercyclic vectors is dense in $X$, we say the operators $T_1, \dots, T_N$ are \textit{densely d-hypercyclic}.

Over the last fifteen years, disjoint hypercyclicity has been studied extensively. Although at first, the dynamics of a single operator and disjoint dynamics of finite family of operators seem to be very similar (see \cite{BG2}, \cite{BeP1}, \cite{BeMaPeSh}, \cite{Sa3}, \cite{Sa4}, and \cite{Shk2}), later work in the area proved this to be false. Especially, the work in \cite{BeMaPe}, \cite{BeMaSa},  \cite{MarPui21}, \cite{MaSa}, and \cite{SaSh} provide several aspects of their differing dynamical properties. Recent work reveal connections between disjoint hypercyclicity and other areas of mathematics such as the notion of multiple recurrence in topological dynamics \cite{CosPar12} and Sidon sets in additive combinatorics  \cite{Car23}. Also, see \cite{Bay23}, \cite{MarMenPui22} and \cite{MarPui21} for the study of disjointness among families of frequently hypercyclic operators. 

In the present paper, we answer two questions posed by Salas \cite{Sa4} about disjoint hypercyclic operators. The first question posed by Salas focuses on the existence of disjoint hypercyclic operators. It has been shown independently by many authors that, given $N \geq 2$, every separable infinite dimesional Fr\'echet space supports disjoint hypercyclic operators $T_1, \dots, T_N$ (see \cite{BeMaPeSh}, \cite{Sa3}, and \cite{Shk2}). Salas raised the following natural question \cite[Question 3]{Sa4}: 

\begin{question}\label{Q1}
    Given d-hypercyclic operators $T_1, \dots, T_N \in \mathcal{B}(X)$, can one select an additional operator $T_{N+1}$ such that extended family of operators $T_1, \dots, T_N, T_{N+1}$ remains d-hypercyclic? 
\end{question}

In \cite{MaSa}, the authors gave a partial answer to the disjoint supercyclic version of the question for families of weighted shift operators. They showed that if the weighted shifts $B_1, \dots, B_N$ are d-supercyclic, one can always add another operator $T_{N+1}$ such that $B_1, \dots, B_N, T_{N+1}$ are still d-supercyclic although they also proved that it is not always possible to choose this additional operator as another weighted shift. 

In Section \ref{Sec1}, we give a complete positive answer to Question \ref{Q1}. In fact, we show that the set of operators $T_{N+1}$ which extend the family to a new tuple of $N+1$ of d-hypercyclic operators is dense in the strong operator topology (or, SOT-dense for short) in the algebra of operators $\mathcal{B}(X)$. The result also enables us to show the existence of two densely d-hypercyclic operators which fail to satisfy the Disjoint Blow-up/Collapse property, and it provides an alternative technique to establish the existence of d-hypercyclic operators on every separable infinite dimensional Fr\'echet space admitting a continuous norm.

In the rest of the paper, we focus on the set of d-hypercyclic vectors for finite families of operators. It is easy to see that if an operator $T$ supports a hypercyclic vector, then the set of hypercyclic vectors $\mathcal{HC}(T)$ is a dense $G_\delta$. Furthermore, there is dense hypercyclic manifold in which every nonzero vector in the manifold is a hypercyclic vector for the operator \cite[Theorem 1.30]{BayMat}.  On the positive side, Salas \cite{Sa4} defined a stronger version of the Disjoint Blow-up/Collapse property (see \cite[Definition 1]{Sa4}), and he showed this stronger version provided a sufficient condition for the existence of a dense manifold of d-hypercyclic vectors. Moreover, he showed that for finite families of weighted shift operators, the Disjoint Blow-up/Collapse property and its strong version are equivalent. This led Salas to raise the following question \cite[Question 1]{Sa4}: 

\begin{question}\label{Q2}
    Are the Disjoint Blow-up/Collapse property and its stronger version equivalent?
\end{question}

In Section \ref{Sec2}, we provide a counterexample consisting of two d-weakly mixing operators whose set of d-hypercyclic vectors is dense but every d-hypercyclic manifold is one-dimensional. Since d-weakly mixing operators satisfy the Disjoint Blow-up/Collapse property, these operators provide a negative answer of Question \ref{Q2}. 

\section[Extending d-Hypercyclic Families]{Extending d-Hypercyclic Families}\label{Sec1}

In \cite{Ch}, Chan proved that the hypercyclic operators on a separable infinite dimensional Hilbert space form a SOT-dense subset of the algebra of continuous linear operators. This result was later extended further to hypercyclic operators on Fr\'echet spaces by B\`es and Chan \cite{BeCh}.
Motivated by Chan's result, we answer Question \ref{Q1} in the positive by showing that for any finite family of d-hypercyclic operators $T_1, \ldots, T_N \in \mathcal{B}(X)$, the collection of operators which extend the operators $T_1, \ldots, T_N$ while maintaining d-hypercyclicity is SOT-dense in $\mathcal{B}(X)$. Even more, we are able to maintain some control over the d-hypercyclic vectors in the extended family of d-hypercyclic operators.

\begin{theorem}\label{Extending}
	Let $T_1, \dots, T_N \in \mathcal{B}(X)$ with $N \in \mathbb{N}$ be d-hypercyclic operators, and let $\mathcal{A} \subseteq \mbox{d-}\mathcal{HC}(T_1, \dots, T_N)$.  If the set $\mathcal{A}$ is countable and linearly independent, then the collection
    \begin{align*}
        \left\{ T \in \mathcal{B}(X) : \mathcal{A} \subseteq \mbox{d-}\mathcal{HC} (T_1, \dots, T_N, T) \right\}
    \end{align*}
    is SOT-dense in $\mathcal{B}(X)$.
\end{theorem}

For a given hypercyclic operator $B \in \mathcal{B}(X)$, it was noted in \cite{BeCh1} that its similarity orbit
\begin{align}
    \mathcal{S}(B) = \{ L^{-1} B L : \mbox{$L \in \mathcal{B}(X)$ invertible} \}
    \nonumber
\end{align}
consists entirely of hypercyclic operators and is SOT-dense in $\mathcal{B}(X)$.  
The proof of Theorem \ref{Extending} takes advantage of the relationship between a hypercyclic operator and its conjugates.
The density portion of Theorem \ref{Extending} utilizes an SOT-dense subset within the similarity orbit $\mathcal{S}(B)$.

\begin{lemma}\label{SOT-dense}
    Let $B \in \mathcal{B}(X)$ be a hypercyclic operator, and let $\mathcal{A}_1, \mathcal{A}_2$ be two countable dense and linearly independent sets in $X$.
    The collection 
    \begin{align*}
        \mathcal{C} = \left\{ L^{-1} B L: L \in \mathcal{B}(X) \mbox{ invertible and } L(\mathcal{A}_1) = \mathcal{A}_2  \right\}
    \end{align*}
    of hypercyclic operators is SOT-dense in in $\mathcal{B}(X)$.
\end{lemma}

\begin{proof}
    For the hypercyclic operator $B \in \mathcal{B}(X)$, Lemma \ref{SOT-dense} follows after showing the collection $\mathcal{C}$ is norm dense within the SOT-dense similarity orbit $\mathcal{S}(B)$, which consists entirely of hypercyclic operators.  
    Let $L \in \mathcal{B}(X)$ be any invertible operator. In \cite[Lemma 2.1]{Gri}, Grivaux showed that for any $\epsilon > 0$ and any two countable dense linearly independent sets $\widetilde{\mathcal{A}}_1, \widetilde{\mathcal{A}}_2$ in $X$, there exists an invertible operator 
    $J \in \mathcal{B}(X)$ for which 
    $J(\widetilde{\mathcal{A}}_1) = \widetilde{\mathcal{A}}_2$ and $\| J - I \| < \epsilon$.  
    From this result, we may select a sequence $(J_n)_{n = 1}^{\infty}$ of invertible operators in $\mathcal{B}(X)$ with $J_n(\mathcal{A}_1) = L^{-1}(\mathcal{A}_2)$ and $\| J_n - I \| \longrightarrow 0$ as $n \longrightarrow \infty$.  Furthermore, since the mapping $T \mapsto T^{-1}$ is norm-continuous on the invertible operators in $\mathcal{B}(X)$, it follows $\| J_n^{-1} - I \| \longrightarrow 0$ as $n \longrightarrow \infty$. Consider the sequence $(L_n)_{n = 1}^{\infty}$ given by $L_n = L J_n$.  Each invertible operator $L_n$ satisfies $L_n(\mathcal{A}_1) = L J_n(\mathcal{A}_1) = \mathcal{A}_2$, and so $L_n^{-1} B L_n \in \mathcal{C}$.  
    Furthermore, by the convergence of the sequence $(J_n)_{n = 1}^{\infty}$, we get 
    $\| L_n^{-1} B L_n - L^{-1}B L \| = \| J_n^{-1}(L^{-1} B L) J_n - L^{-1} B L \| \longrightarrow 0$ 
    as $n \longrightarrow \infty$.
\end{proof}

The proof that each conjugate in the similarity orbit of a hypercyclic operator is hypercyclic follows from the
well-known fact that $f$ is a hypercyclic vector for a conjugate $L^{-1} B L$ if and only $Lf$ is a hypercyclic vector for the operator $B$.  This idea can be extended to d-hypercyclic vectors for a finite family of conjugates. 

\begin{lemma}\label{Conjugates}
    Let $B_1, \dots, B_M, L_1, \dots, L_M \in \mathcal{B}(X)$ with $M \in \mathbb{N}$ and the operators $L_1, \dots, L_M$ invertible.  For each integer $i$ with $1\le i \le M$, let $A_i = L_i^{-1} B_i L_i$.  The vector $f \in \mbox{d-}\mathcal{HC}(A_1, \dots A_M)$ if and only if $(L_1 f, \dots, L_M f) \in \mathcal{HC}(B_1 \oplus \cdots \oplus B_M)$.
\end{lemma}

\begin{proof}
    Variations of Lemma \ref{Conjugates} appear throughout the literature on disjoint hypercyclic. For example, see \cite[Lemma 2.1]{SaSh}.  The lemma follows from the observation that for any vectors $f, x_1, \dots, x_M \in X$, we have $B_i^{n_k} (L_i f) \longrightarrow x_i$ as $k \longrightarrow \infty$ if and only $A_i^{n_k} f = (L_i^{-1} B_i L_i)^{n_k} f = L_i^{-1} B_i^{n_k}(L_i f) \longrightarrow L_i^{-1} x_i$ as $k \longrightarrow \infty$.
\end{proof}

Theorem \ref{Extending} is established using both of these lemmas.

\begin{proof}[Proof of Theorem \ref{Extending}]
    With Lemma \ref{SOT-dense} in mind, it suffices to find a hypercyclic operator $B \in \mathcal{B}(X)$ and two countable dense and linearly independent sets
    $\mathcal{A}_1, \mathcal{A}_2$ such that
    \begin{align}
        \mathcal{A} \subseteq \mbox{d-}\mathcal{HC}(T_1, \dots, T_n, L^{-1}BL)
        \mbox{ \ \ whenever $L \in \mathcal{B}(X)$ is invertible and $L(\mathcal{A}_1) = \mathcal{A}_2$.}
        \label{Extending_01}
    \end{align}
    For the hypercyclic operator $B \in \mathcal{B}(X)$, recall that every separable, infinite dimensional Banach space admits a mixing operator
    (see Ansari \cite{A}, Bernal \cite{BG}, or Grivaux \cite{Gri2}).  Let $B$ be any such mixing operator in $\mathcal{B}(X)$.  For the set $\mathcal{A}_1$, select any countable dense and linearly independent set $\mathcal{A}_1$ with $\mathcal{A} \subseteq \mathcal{A}_1$.  In order to define set $\mathcal{A}_2$, let
    \begin{align}
      \{ (g_{1,j}, \dots, g_{N,j}, g_{N+1,j}) : j \in \mathbb{N} \}
      \label{Extending_02}
    \end{align}
    be a countable dense set in $\bigoplus_{i=1}^{N+1} X$. Since $\mathcal{A} \subseteq \mbox{d-}\mathcal{HC}(T_1, \dots, T_N)$, for each $f \in \mathcal{A}$ and each $j \in \mathbb{N}$, there exists a strictly increasing sequence $(n_{k,j}^{(f)})_{k=1}^{\infty}$ of positive integers such that
    \begin{align}
	   T_{i}^{n_{k,j}^{(f)}} f \longrightarrow g_{i,j} \mbox{ as $k \longrightarrow \infty$ for each $i$ with $1 \le i \le N$.}
	   \label{Extneding_03}
    \end{align}
    A mixing operator in hereditarily hypercyclic with respect to the full sequence $(n)_{n = 1}^{\infty}$, and so
    the set $\mathcal{HC} (\{ B^{n_{k,j}^{(f)}} \}_{k = 1}^{\infty})$ is a dense $G_{\delta}$ for each $f \in \mathcal{A}$ and each $j \in \mathbb{N}$ (see B\`es, Peris \cite{BeP}).
    It follows from the Baire Category Theorem that the countable intersection
	$\bigcap_{f \in \mathcal{A}} \bigcap_{j = 0}^{\infty}\mathcal{HC} (\{ B^{n_{k,j}^{(f)}} \}_{k = 1}^{\infty})$
	is also a dense $G_{\delta}$.  Let $\mathcal{A}_2$ be any countable dense and linear independent set with
    \begin{align}
        \mathcal{A}_2 \subseteq \bigcap_{f \in \mathcal{A}} \bigcap_{j = 0}^{\infty}\mathcal{HC} (\{ B^{n_{k,j}^{(f)}} \}_{k=1}^{\infty}).
        \nonumber
    \end{align}
    
    To complete the proof it remains to show the condition in (\ref{Extending_01}) is satisfied.  Let $L \in \mathcal{B}(X)$ be an invertible operator with $L(\mathcal{A}_1) = \mathcal{A}_2$, and let 
    $f \in \mathcal{A}$.  
    Since
    $L f \in \mathcal{A}_2 \subseteq \bigcap_{f \in \mathcal{A}} \bigcap_{j = 0}^{\infty}\mathcal{HC} (\{ B^{n_{k,j}^{(f)}} \}_{k = 1}^{\infty})$, for each $j \in \mathbb{N}$, there exists a strictly increasing subsequence $(n_{k_r,j}^{(f)})_{r = 1}^{\infty}$ of positive integers for which
    $B^{n_{k_r,j}^{(f)}} (L f) \longrightarrow g_{N+1, j}$ as $r \longrightarrow \infty$.
    Furthermore, it follows from (\ref{Extneding_03}) that
    $T_{i}^{n_{k_r,j}^{(f)}} f \longrightarrow g_{i,j}$ as $r \longrightarrow \infty$ for each integer $i$ with $1 \le i \le N$.
    Therefore, for each $j \in \mathbb{N}$, we have
    \begin{align}
        \left( T_{1} \oplus \dots \oplus T_{N} \oplus B \right)^{n_{k_r,j}^{(f)}} (f, \dots, f, Lf) 
        \longrightarrow 
        (g_{1,j}, \dots, g_{N,j}, g_{N+1, j} ) \mbox{ \ as $r \longrightarrow \infty$.}
        \nonumber
    \end{align}
    Since the set $\{ (g_{1,j}, \dots, g_{N,j}, g_{N+1,j}) : j \in \mathbb{N} \}$ is dense, 
    it follows $(f, \dots, f, Lf) \in \mathcal{HC}(T_1 \oplus \cdots \oplus T_N \oplus B)$, and so $f \in \mbox{d-}\mathcal{HC}(T_1, \dots, T_N, L^{-1}B L)$ by Lemma \ref{Conjugates}.
\end{proof}

From Theorem \ref{Extending}, any finite collection of $N$ d-hypercyclic operators $T_1, \dots, T_N$ may be extended to $N+1$ d-hypercyclic operators $T_1, \dots, T_N, T_{N+1}$.  When the original operators $T_1, \dots, T_N$ are densely d-hypercyclic, the operator $T_{N+1}$ may be selected so the extended collection $T_1, \dots, T_N, T_{N+1}$ remains densely d-hypercyclic by simply selecting the countable linear independent $\mathcal{A}$ in Theorem \ref{Extending} to also be dense. 

\begin{corollary}\label{D}
	If $T_1, \dots, T_N \in \mathcal{B}(X)$ with $N \in \mathbb{N}$ are densely d-hypercyclic operators, then the collection
    \begin{align*}
        \{ T \in \mathcal{B}(X) : \mbox{ $T_1, \dots, T_N, T$ are densely d-hypercyclic }\}
    \end{align*}
    is SOT-dense in $\mathcal{B}(X)$.  
    Moreover, for any countable and linearly independent set 
    $\mathcal{A} \subseteq \mbox{d-}\mathcal{HC}(T_1, \dots, T_N)$, there is an SOT-dense subcollection of operators for which $\mathcal{A} \subseteq \mbox{d-}\mathcal{HC}(T_1, T_2, \dots, T_N,T)$.
\end{corollary}

Theorem \ref{Extending} or Corollary \ref{D} may be applied to generalize results about two d-hypercylic operators $T_1, T_2$ to a finite family of d-hypercyclic operators $T_1, \dots, T_N$ with $N \ge 2$.  For an example, given any nonzero vector $g \in X$, it was shown in \cite[Corollary 3.6]{SaSh} that there exists d-hypercylic $T_1, T_2$ and densely d-hypercyclic operators $\widetilde{T}_1, \widetilde{T}_2$ such that
$\mbox{d-}\mathcal{HC}(T_1, T_2) = \mbox{span} \{ g \} \setminus \{ 0 \}$ and 
$\mbox{d-}\mathcal{HC}(T_1, T_2) \cap \mbox{d-}\mathcal{HC}( \widetilde{T}_1, \widetilde{T}_2) = \emptyset$.  
This result showed that a set of d-hypercyclic can be extremely small and nowhere dense.
By repeatedly applying Theorem \ref{Extending} with operators $T_1, T_2$ and $\mathcal{A} = \{ g \}$ and repeatedly applying Corollary \ref{D} with operators $\widetilde{T}_1, \widetilde{T}_2$, we can construct d-hypercyclic operators $T_1, \dots, T_N$ and densely d-hypercyclic operators
$\widetilde{T}_1, \dots \widetilde{T}_N$ with $\mbox{d-}\mathcal{HC}(T_1, \dots, T_N) = \mbox{span} \{ g \} \setminus \{ 0 \}$ and 
$\mbox{d-}\mathcal{HC}(T_1,\dots, T_N) \cap \mbox{d-}\mathcal{HC}( \widetilde{T}_1, \dots,  \widetilde{T}_N) = \emptyset$.

A single hypercyclic operator is always densely hypercyclic but may fail to be weakly mixing \cite{DeRe}. On the other hand, being weakly mixing, satisfying the Hypercyclicity Criterion, and satisfying the Blow-up/Collapse property are all equivalent (see \cite{BeP} and \cite{Le}). In the disjoint setting, none of these equivalences hold. Even the stronger notion of frequent d-hypercyclicity (see \cite{MarMenPui22} or \cite{MarPui21} for the definition) does not ensure being d-weakly mixing contrary to the single operator situation \cite[Theorem 9.8]{GroPer}. The below chart summarizes these varying results regarding disjoint hypercyclic operators in the chart below with references:

\[
\begin{matrix}
    & \ & \ \mbox{d-Hypercyclicity Criterion}   & \ & \ 
    \\
    & \ & \ \Downarrow \quad \not\Uparrow      & \ & \mbox{\cite{SaSh}}   
    \\  
    \mbox{frequently d-hypercyclic} &  
        \begin{aligned} 
            \nLeftarrow & \\ 
            \nRightarrow 
        \end{aligned}  
    & \ \mbox{d-weakly mixing } & \ & \mbox{\cite{MarPui21}} 
    \\
    & \ & \ \Downarrow \quad \not\Uparrow      & \ & \mbox{\cite{BeMaSa}}   
    \\ 
    & \ & \ \mbox{Disjoint Blow-up/Collapse} & \  & \ 
    \\
    & \ & \ \Downarrow  & & \mbox{\cite{BeP1}} 
    \\
    & \ & \ \mbox{densely d-hypercyclic} & \  & \ 
    \\
    & \ & \ \Downarrow \quad \not\Uparrow & \ & \mbox{\cite{SaSh}} 
    \\
    & \ & \ \mbox{d-hypercyclic} & \  & \ 
    \\
    & \ & \
\end{matrix}
\]

In \cite[Theorem 3.1]{Bay23}, Bayart provided an example of two frequently d-hypercyclic unilateral weighted shifts on the sequence space $\ell^1(\N)$. As weighted shifts are never d-weakly mixing \cite[Proposition 3.2]{BeMaSa}, these weighted shifts give a constructive counterexample of frequently d-hypercyclic operators which fail to be d-weakly mixing.
Corollary \ref{D} provides a counterexample for the missing inverse implication between dense d-hypercyclicity and the Disjoint Blow-up/Collapse property

\begin{corollary} 
    For each integer $N \ge 2$, there exists densely d-hypercyclic operators $T_1, \dots, T_N$ which fail to satisfy the Disjoint Blow-up/Collapse property.
\end{corollary}

\begin{proof}
    In \cite{DeRe}, it is shown that there exists a (densely) hypercyclic operator $T_1$ which fails to be weakly mixing, and therefore, fails to satisfy the Blow-up/Collapse property. Starting with this operator $T_1$ and repeatedly applying Corollary \ref{D}, we extend to a finite family of densely d-hypercyclic operators $T_1, \dots, T_N$. Since the original operator $T_1$ fails to satisfy the Blow-up/Collapse property, the extended family of operators $T_1, \dots, T_N$ must also fail to satisfy the Disjoint Blow-up/Collapse property.
\end{proof}

Other implications are missing in the above chart. Obviously, frequently d-hypercyclic operators are d-hypercyclic but, as far as we know, the following problems are still open:

\begin{question}\label{Q3}
    Are frequently d-hypercyclic operators also densely d-hypercyclic? Do frequently d-hypercyclic operators satisfy the Disjoint Blow-up/Collapse property?
\end{question}

Disjoint hypercyclicity is concerned with the phenomenon of concurrent approximations using orbits of differing operators with the same initial vector. A similar but weaker notion in the literature was introduced recently by Bernal and Jung \cite{BGJ}.

\begin{definition} \cite[Definition 2.1]{BGJ}\label{D:s-hyper} 
    {\rm
	Operators $T_1, \ldots, T_N \in \mathcal{B}(X)$ with $N \ge 2$ are called \textit{simultaneously hypercyclic}, or \textit{s-hypercyclic} for short, if there exists a vector $f \in X$ such that 
	\begin{align*}
		\Delta\Bigl( \bigoplus_{i=1}^{N} X \Bigr) \subseteq \overline{\{(T_{1}^n f,\ldots,T_{N}^n f):n \in \mathbb{N} \} },
	\end{align*}
	where $\Delta\left( \bigoplus_{i=1}^{N} X \right) = \left\{(x,\ldots,x): x \in X \right\}$ denotes the diagonal of $\bigoplus_{i=1}^{N} X$. Such a vector $f \in X$ is called a \textit{s-hypercyclic vector} of $T_1, \ldots, T_N$. If the set 
    $\mbox{s-} \mathcal{HC}(T_1, \dots, T_N)$ of s-hypercyclic vectors is dense in $X$, then the operators $T_1, \ldots, T_N$ are called as \textit{densely s-hypercyclic}.
    }
\end{definition}

The proof of Theorem \ref{Extending} may be modified to show any finite family of s-hypercyclic operators (or densely s-hypercyclic operators) $T_1, \dots, T_N$ can also be extended.  For s-hypercyclicity, it is only required that the diagonal $\{ (x, \dots, x) : x \in X \}$ is contained within the closure of the orbit $\mbox{orb}( T_1 \oplus \cdots \oplus T_N, (f, \dots, f))$.  
To adjust the proof of Theorem \ref{Extending} for the s-hypercyclicity situation, replace the countable dense set 
$\{ (g_{1,j}, \dots, g_{N,j}, g_{N+1, j}) : j \in \mathbb{N} \}$ in (\ref{Extending_02}) with a countable set
\begin{align*}
    \{ (g_j, \dots, g_j, g_j) : j \in \mathbb{N} \},
\end{align*}
which is dense in the diagonal $\Delta\left( \bigoplus_{i=1}^{N+1} X \right)$ and replace each subsequent vector $g_{i,j}$ with the vector $g_j$.  With these minor modifications within the proof, we obtain following simultaneous hypercyclicity version of Theorem \ref{Extending}.

\begin{corollary}\label{s-hyper}
    Let $T_1, \dots, T_N \in \mathcal{B}(X)$ with $N \in \mathbb{N}$ be s-hypercyclic operators, and let $\mathcal{A} \subseteq \mbox{s-}\mathcal{HC}(T_1, \dots, T_N)$.  If the set $\mathcal{A}$ is countable and linearly independent, then the collection
    \begin{align*}
        \left\{ T \in \mathcal{B}(X) : \mathcal{A} \subseteq \mbox{s-}\mathcal{HC} (T_1, \dots, T_N, T) \right\}
    \end{align*}
    is SOT-dense in $\mathcal{B}(X)$.
\end{corollary}

As with Corollary \ref{D}, when the original operators $T_1, \dots, T_N$ are densely s-hypercyclic, the operator $T_{N+1}$ may be selected so the extended collection $T_1, \dots, T_N, T_{N+1}$ remains densely s-hypercyclic by simply selecting the countable linear independent $\mathcal{A}$ in Corollary \ref{s-hyper} to also be dense.

\begin{remark}
    {\rm 
    The proof of Theorem \ref{Extending} does not depend on the Banach space structure of the underlying space except for the result of Grivaux \cite[Lemma 2.1]{Gri} used in Lemma \ref{SOT-dense}. But Grivaux's result has already been generalized to Fr\'echet spaces that admit a continuous norm by Albanese \cite{Al}. Therefore, the results in Section \ref{Sec1} also hold for these spaces. In \cite{BP}, Bonet and Peris show that every separable infinite dimensional Fr\'echet space supports a hypercyclic operator, and such spaces always support mixing operators (see \cite[Theorem 8.9]{GroPer}). Thus, for any Fr\'echet space that admits a continuous norm, we can choose a (densely) hypercyclic operator $T_1$ acting on this space and add operators $T_2, \ldots, T_N$ iteratively using Corollary \ref{D} such that $T_1, \ldots, T_N$ are densely d-hypercyclic. This process provides a new method to showing that any Fr\'echet spaces with a continuous norm support densely d-hypercyclic operators $T_1, \ldots, T_N$ for any integer $N \ge 2$.
    }
\end{remark}


\section[d-Hypercyclic Manifolds]{d-Hypercylic Subspaces and Manifolds}\label{Sec2}

The primary aim of Section \ref{Sec2} is to answer Question \ref{Q2} in the negative by providing an example of operators which satisfy the Disjoint Blow-up/Collapse property but fail to satisfy its strong version defined by Salas \cite[Definition 1]{Sa4}. In order to achieve this goal, we first prove the following theorem about generating dense d-hypercylic operators connected to a common closed subspace.

\begin{theorem}\label{d-subspace}
    Let $Y$ be a closed subspace of the Banach space $X$ with codimension $\mbox{dim}(X / Y) \ge M$.
    If the operator $T_1 \in \mathcal{B}(X)$ satisfies $Y \setminus \{ 0 \} \subseteq \mathcal{HC}(T_1)$ and satisfies the Hypercyclicity Criterion, 
    then there exists an operator $T_2 \in \mathcal{B}(X)$ for which the following hold: 
    \begin{enumerate}
    \item[(i)]  
    The direct sum operators $\bigoplus_{i=1}^{M} T_1, \bigoplus_{i=1}^{M} T_2$ are densely d-hypercyclic;
    
    \item[(ii)]  
    We have $Y \setminus \{ 0 \} \subseteq \mathcal{HC} (T_1) \cap \mathcal{HC} (T_2)$ and $Y \cap \mbox{d-} \mathcal{HC}(T_1,T_2) = \emptyset$;
    
    \item[(iii)]  
    The set d-$\mathcal{HC}(T_1,T_2) \cup \{ 0 \}$ contains a subspace of dimension $M$ but fails to contain any subspace of dimension greater than $M$.
\end{enumerate}
\end{theorem}

\begin{proof}
    To construct the desired operator $T_2 \in \mathcal{B}(X)$, first note the codimension $\mbox{dim}(X/Y) \ge M$ means we may select a linearly independent set $\{ z_1, \dots, z_M \}$ in $X$ such that
    \begin{align}
        Y \cap \mbox{span} \{ z_1, \dots, z_M \} = \{ 0 \}.
        \nonumber
    \end{align}
    For each integer $i$ with $1 \le i \le M$, consider the closed subspace
    \begin{align}
        Y_i = Y \oplus \mbox{span} \{ z_j : j \ne i \}.
    \label{d-subspace_02}
    \end{align}
    For each integer $i$ with $1 \le i \le M$, note that $z_i \not\in Y_i$, and so by a standard corollary of the Hahn Banach Theorem
    \cite[Corollary 6.8]{Co}, there exists a linear functional $\lambda_i$ on $X$ for which
    \begin{align}
        \lambda_i(z_i) = 1 \mbox{ \ and \ } \lambda_i(x) = 0 \mbox{ \ for all $x \in Y_i$}.
        \label{d-subspace_03}
    \end{align}
    Next, since the operator $T_1$ satisfies the Hypercyclicity Criterion, the direct sum operator $\bigoplus_{i = 1}^{K} T_1$ is hypercyclic for each $K \in \mathbb{N}$ (see \cite[Theorem 2.1]{BGGE}).  Select any hypercyclic vector
    \begin{align}
        (h_1, \dots, h_M, g_1, \dots, g_M) \in \mathcal{HC} \biggl( \bigoplus_{i = 1}^{2M} T_1 \biggr).
        \label{d-subspace_04}
    \end{align}
    Finally, select a scalar $\alpha$ satisfying $0 < \alpha < \left( \sum_{j = 1}^{M} \| \lambda_j \| \| g_j \| \right)^{-1}$ and define the operator $L:X \longrightarrow X$ by
    \begin{align}
        L x = x + \alpha \sum_{j = 1}^{M} \lambda_j(x) g_j
        \nonumber
    \end{align}
    By the selection of $\alpha$, we get $\| L- I \| < 1$, which in turn implies the operator $L \in \mathcal{B}(X)$ is invertible.  
    Set $T_2 = L^{-1} T_1 L$.

    The verification that the direct sum operators $\bigoplus_{i = 1}^{M} T_1, \bigoplus_{i = 1}^{M} T_2$ are densely d-hypercyclic is similar to an argument presented in \cite[Theorem 2.3]{SaSh}.  Consider the set 
    \begin{align*}
        \Gamma = \left\{ (f_1, \dots, f_M) \in \bigoplus_{i = 1}^{M} X :  (f_1, \dots, f_M, g_1, \dots, g_M) \in 
        \mathcal{HC} \left( \bigoplus_{i = 1}^{2M} T_1 \right) \right\}.
    \end{align*}
    By the selection of the vector in (\ref{d-subspace_04}), one can show 
    $(T_1^{k_1} h_1, \dots, T_1^{k_M} h_M, g_1, \dots, g_M) \in \mathcal{HC} \left( \bigoplus_{i = 1}^{2M} T_1 \right)$ for any positive integers
    $k_1, \dots, k_M$.  
    Hence, the set $\Gamma$ is dense in $\bigoplus_{i = 1}^{M} X$ since it contains the dense set $\mbox{Orb}(T_1, h_1) \oplus \cdots \oplus \mbox{Orb}(T_1, h_M)$.
    Next, define the operator $A: \bigoplus_{i = 1}^{M} X \longrightarrow \mathcal{B}(\mathbb{F}^M)$ by
    \begin{align}
        A(x_1,x_2, \dots, x_M) =
        \left[ 
        \begin{array}{cccc}
            \lambda_{1}(x_1) & \lambda_{1}(x_2)  & \cdots & \lambda_{1}(x_M) \\
            \lambda_{2}(x_1) & \lambda_{2}(x_2)  & \cdots & \lambda_{2}(x_M) \\
            \vdots & \vdots  & \ddots & \vdots \\
            \lambda_{M}(x_1) & \lambda_{M}(x_2)  & \cdots & \lambda_{M}(x_M) \\
        \end{array}
        \right].
        \nonumber
    \end{align}
    Based on the definition of the linear functionals $\lambda_1, \dots, \lambda_N$ in (\ref{d-subspace_03}), one can easily verify the operator $A$ is onto. Since the set $\mathcal{O}$ of invertible matrices is open and dense in $\mathcal{B}(\mathbb{F}^M)$, the set $A^{-1}(\mathcal{O})$ is open and dense in $\bigoplus_{i = 1}^{M} X$.  The direct sum operators $\bigoplus_{i=1}^{M} T_1, \bigoplus_{i=1}^{M} T_2$ are densely d-hypercyclic once we establish the dense set
    $\Gamma \cap A^{-1}(\mathcal{O}) \subseteq \mbox{d-}\mathcal{HC}(\bigoplus_{i = 1}^{M} T_1, \bigoplus_{i = 1}^{M} T_2)$.
    To this end, let $(f_1, \dots, f_M) \in \Gamma \cap A^{-1}(\mathcal{O})$.  For any vectors $x_1, \dots, x_M, y_1, \dots, y_M \in X$, there exists a strictly increasing sequence $(n_k)_{k=1}^{\infty}$ of positive integers such that for integers $j, \ell$ with $1 \le j, \ell \le M$,
    \begin{align}
        T_1^{n_k}(f_{\ell}) & \longrightarrow x_{\ell}
        \mbox{ \ \ and \ \ }
        T_1^{n_k} g_{j}  \longrightarrow \frac{1}{\alpha} \sum_{i = 1}^{M} \beta_{i,j} (y_{i} - x_i) \mbox{ as $k \longrightarrow \infty$},
        \nonumber
    \end{align}
    where $[\beta_{i,j}]_{i, j = 1,\dots, M}$ is the inverse matrix $[A( f_1, \dots, f_M)]^{-1}$.  
    For any integer $\ell$ with $1 \le \ell \le M$, we have
    \allowdisplaybreaks
    \begin{align}
        T_1^{n_k}(L f_{\ell})
        & = 
        T^{n_k} \Bigl( f_{\ell} + \alpha \sum_{j = 1}^{M} \lambda_j(f_{\ell}) g_j \Bigr)
        \nonumber \\
        &=
        T_1^{n_k}f_{\ell} + \alpha \sum_{j = 1}^{M} \lambda_j(f_{\ell}) T_1^{n_k} g_j
        \nonumber \\
        & \longrightarrow
        x_{\ell} + \alpha \sum_{j = 1}^{M} \lambda_j(f_{\ell}) \Bigl( \frac{1}{\alpha} \sum_{i = 1}^{M} \beta_{i,j}(y_i - x_i)\Bigr)
        \mbox{ \ \ as $k \longrightarrow \infty$}
        \nonumber \\
        & =
        x_{\ell} + \sum_{i = 1}^{M} \sum_{j = 1}^{M} \beta_{i, j} \lambda_{j}(f_{\ell})(y_i - x_i)
        \nonumber \\
        & = 
        x_{\ell} + (y_{\ell} - x_{\ell})
        \nonumber \\
        & = y_{\ell}.
        \nonumber
    \end{align}
    Hence, 
    \begin{align*}
        \left( \bigoplus_{i = 1}^{2M} T_1 \right)^{n_k}(f_1, \dots, f_M, Lf_1, \dots, Lf_M)
        \longrightarrow (x_1, \dots, x_M, y_1, \dots, y_M)
        \mbox{ \ \ as $k \longrightarrow \infty$.}
    \end{align*}
    We conclude that $(f_1, \dots, f_M, Lf_1, \dots, Lf_M) \in \mathcal{HC}(\bigoplus_{i = 1}^{2M} T_1)$.
    Applying Lemma \ref{Conjugates} with the direct sum operators $\bigoplus_{i = 1}^{M} T_1$, $\bigoplus_{i = 1}^{M} T_1$ and their conjugates 
    $\bigoplus_{i = 1}^{M} T_1$, $\bigoplus_{i = 1}^{M} T_2 = \bigoplus_{i = 1}^{M} L^{-1}T_1L$ gives us that the vector  
    $(f_1, \dots, f_M) \in \mbox{d-}\mathcal{HC}(\bigoplus_{i = 1}^{M} T_1, \bigoplus_{i = 1}^{M} T_2)$.
    
    To establish $Y \setminus \{ 0 \} \subseteq \mathcal{HC}(T_1) \cap \mathcal{HC}(T_2)$ and $Y \cap \mbox{d-}\mathcal{HC}(T_1, T_2) = \emptyset$, observe that for any nonzero vector $y \in Y \setminus \{ 0 \} \subseteq \mathcal{HC}(T_1)$, we have
    \begin{align*}
        L y 
        & = 
        y + \alpha \sum_{j = 1}^{M} \lambda_j(y) g_j
        =
        y, 
        \mbox{ \ \ by (\ref{d-subspace_02}) and (\ref{d-subspace_03}). }
    \end{align*}
    Since $L y = y \in \mathcal{HC}(T_1)$ and $(y, Ly) = (y, y) \not \in \mathcal{HC}(T_1 \oplus T_1)$, it follow 
    $y \in \mathcal{HC}(T_2)$ and $y \not\in \mbox{d-}\mathcal{HC}(T_1, T_2)$ by applying Lemma \ref{Conjugates} with
    $T_1$, $T_2 = L^{-1} T_1 L$.

    To show the set $\mbox{d-}\mathcal{HC}(T_1, T_2) \cup \{ 0 \}$ contains a subspace of dimension $M$, select any d-hypercyclic vector 
    $(h_1, \dots, h_M) \in \mbox{d-}\mathcal{HC}(\bigoplus_{i = 1}^{M} T_1, \bigoplus_{i = 1}^{M} T_2)$ and set 
    $Y_0 = \mbox{span} \{ h_1, \dots, h_M \}$. 
    In order for $(h_1, \dots, h_M)$ to be a d-hypercyclic vector for the direct sum operators 
    $\bigoplus_{i = 1}^{M} T_1, \bigoplus_{i = 1}^{M} T_2$, the vectors $h_1, \dots, h_M$ must be linearly independent, and so the subspace $Y_0$ has dimension $M$.  To verify $Y_0 \subseteq \mbox{d-}\mathcal{HC}(T_1, T_2) \cup \{ 0 \}$, consider a nontrivial linear combination $f = \gamma_1 h_1 + \cdots + \gamma_M h_M$ in $Y_0$.  Without loss of generality, assume $\gamma_1 \ne 0$.  Since $(h_1, \dots, h_M) \in \mbox{d-} \mathcal{HC}( \bigoplus_{i=1}^{M} T_1, \bigoplus_{i=1}^{M} T_2 )$, for any $x_1, x_2$, there exists a strictly increasing sequence $(n_k)_{k=1}^{\infty}$ of positive 
    integers such that
    \begin{align}
        T_1^{n_k} h_1 & \longrightarrow \frac{1}{\gamma_1} x_1 \mbox{ and } T_2^{n_k} h_1  \longrightarrow \frac{1}{\gamma_1} x_2 \mbox{ as $k \longrightarrow \infty$}
        \nonumber
    \end{align}
    and for integers $i$ with $2 \le i \le M$,
    \begin{align}
        T_1^{n_k} h_i & \longrightarrow 0 \mbox{ and } T_2^{n_k} h_i  \longrightarrow0 \mbox{ as $k \longrightarrow \infty$}.
        \nonumber
    \end{align}
    It follows that 
    $T_1^{n_k}(f) = T_1^{n_k}( \gamma_1 h_1 + \cdots + \gamma_M h_M ) \longrightarrow x_1$ and $T_2^{n_k} f = T_2^{n_k}( \gamma_1 h_1 + \cdots + \gamma_M h_M ) \longrightarrow x_2$ as $k \longrightarrow \infty$.  
    Thus, $f = \gamma_1 h_1 + \cdots + \gamma_M h_M$ is a d-hypercyclic vector for $T_1, T_2$.  

    Last, to show the set $\mbox{d-}\mathcal{HC}(T_1, T_2) \cup \{ 0 \}$ fails to contain any subspace with dimension greater than M, it suffices to show that for any $M+1$ linearly independent vectors $x_1, \dots, x_{M+1}$ in $X$, there exists a nontrivial linear combination of these vectors outside of the set $\mbox{d-}\mathcal{HC}(T_1, T_2) \cup \{ 0 \}$.  Take any linearly independent vectors $x_1, \dots, x_{M+1}$ in $X$ and define the vectors $b_1, \dots, b_{M+1}$ in $\mathbb{F}^M$ by
    \begin{align}
        b_i = (\lambda_1(x_i), \lambda_2(x_i), \dots, \lambda_M(x_{i})).
        \nonumber
    \end{align}
    Since we have the $M+1$ vectors $b_1, \dots, b_{M+1}$ in the $M$ dimensional space $\mathbb{F}^M$, the zero vector in $\mathbb{F}^M$ is a nontrivial linear combination of the vectors $b_1, \dots, b_{M+1}$.  That is,
    \begin{align}
        (0, 0, \dots, 0)
        & = 
        \beta_1 b_1 + \beta_2 b_2 + \dots + \beta_{M+1} b_{M+1}
        \label{d-subspace_11} 
        \\
        & = 
        \biggl( \sum_{i = 1}^{M+1} \beta_i \lambda_1 (x_i), \sum_{i = 1}^{M+1} \beta_i \lambda_2 (x_i), \dots, \sum_{i = 1}^{M+1} \beta_i \lambda_{M}(x_{i}) \biggr)
        \nonumber 
    \end{align}
    for some nontrivial collection of scalars $\beta_1, \dots, \beta_{M+1} \in \mathbb{F}$.
    Consider the nontrivial linear combination 
    \begin{align}
        x = \beta_1 x_1 +  \dots + \beta_{M+1} x_{M+1}.
        \nonumber
    \end{align}
    By the definition of the invertible operator $L$, we have
    \begin{align}
        L x
        & = 
        x + \alpha \sum_{j = 1}^{M} \lambda_j(x) g_j
        \nonumber \\
        & = 
        x + \alpha \sum_{j = 1}^{M} \lambda_j \biggl( \sum_{i = 1}^{M+1} \beta_i x_i \biggr) g_j
        \nonumber \\
        & = 
        x + \alpha \sum_{j = 1}^{M}\biggl( \sum_{i = 1}^{M+1} \beta_i \lambda_j ( x_i ) \biggr) g_j
        \nonumber \\
        & = 
        x + \alpha \sum_{j = 1}^{M} 0 \cdot g_j, \mbox{ by (\ref{d-subspace_11})}
        \nonumber \\
        & = 
        x.
        \nonumber
    \end{align}
    Hence, we have $(x, Lx) = (x,x) \not\in \mathcal{HC}(T_1 \oplus T_1)$, and so the vector 
    $x \not\in \mathcal{HC}(T_1, T_2)$ by applying Lemma \ref{Conjugates} with $T_1$, $T_2 = L^{-1} T_1 L$.
\end{proof}

As noted in the proof of Theorem \ref{Extending}, every separable infinite dimensional Banach space supports a mixing operator, and so it satisfy the Hypercyclicity Criterion. Given such a hypercyclic operator $T \in \mathcal{B}(X)$, one can easily find examples of closed subspaces $Y$ which satisfy the assumptions in Theorem \ref{d-subspace}.  

Applications of Theorem \ref{d-subspace} generate interesting distinctions between hypercyclicity and disjoint hypercyclicity.
For example, a hypercyclic subspace of an operator $T \in \mathcal{B}(X)$ is a closed infinite dimensional subspace in which every nonzero vector is a hypercyclic vector for the operator $T$. An application of Theorem \ref{d-subspace} shows that two dense d-hypercyclic operators $T_1, T_2$, which share a common hypercyclic subspace, may still fail to have a d-hypercylic subspace.

\begin{corollary}\label{B}
    Every separable, infinite dimensional Banach space $X$ admits densely d-hypercyclic operators $T_1,T_2$ for which the operators $T_1, T_2$ have a common hypercyclic subspace but fail to possess a d-hypercyclic subspace.  
    Furthermore, we may assume the only subspaces within the set d-$\mathcal{HC}(T_1,T_2) \cup \{ 0 \}$ are one dimensional.
\end{corollary}

\begin{proof}
    Let $T_1 \in \mathcal{B}(X)$ be any operator that satisfies the Hypercyclicity Criterion and possess a hypercyclic subspace $Y$ with codimension dim$(X /Y) \ge 1$.  See \cite[Theorem 10.28]{GroPer} for an example of such an operator. 
    Applying Theorem \ref{d-subspace} with $M = 1$ yields an operator $T_2$ such that the operators $T_1, T_2$ are densely d-hypercyclic and $Y$ is a hypercyclic subspace for both $T_1$ and $T_2$. 
    Moreover, the operator $T_2$ is selected so that the only subspaces contained within 
    $\mbox{d-}\mathcal{HC}(T_1, T_2) \cup \{ 0 \}$ are one dimensional. Since $\mbox{d-}\mathcal{HC}(T_1, T_2) \cup \{ 0 \}$ fails to contain any infinite dimensional subspace, these densely d-hypercyclic operators fail to possess a d-hypercyclic subspace.
\end{proof}

It is well-known that a hypercyclic operator $T \in \mathcal{B}(X)$ always admits a dense hypercyclic manifold.  That is, a dense linear manifold in which every nonzero vector in the manifold is hypercyclic vector for the operator $T$ (see Herrero \cite{He} and Bourdon \cite{Bd} for the complex case and B\`es \cite{Be} for the real case).  In the disjoint setting, 
clearly being densely d-hypercylic is necessary condition for the existence of a dense d-hypercyclic manifold, especially since
there are examples of non-dense d-hypercyclic operators (see \cite[Corollary 3.5]{SaSh}).
An application of Theorem \ref{d-subspace} shows dense d-hypercyclicity fails to be sufficient condition to ensure the existence of a dense d-hypercylic manifold.

\begin{corollary}\label{A}
    For each integer $N \ge 2$, every separable infinite dimensional Banach space $X$ admits densely d-hypercyclic operators $T_1, \dots, T_N$ that fail to possess a dense d-hypercyclic manifold.  
    Furthermore, we may assume the only subspaces contained within the set d-$\mathcal{HC}(T_1,\dots, T_N) \cup \{ 0 \}$ are one dimensional. 
\end{corollary}

\begin{proof}
    Let $T_1$ be any operator in $\mathcal{B}(X)$ which satisfies the Hypercyclicity Criterion, and let $Y = \mbox{span} \{ f \}$, where $f$ is a hypercyclic vector of $T_1$.   Since the codimension $\mbox{dim}(X/Y) = \infty$, applying Theorem \ref{d-subspace} with $M = 1$ yields an operator $T_2$ for which $T_1, T_2$ are densely d-hypercyclic and for which the only subspaces contained within $\mbox{d-}\mathcal{HC}(T_1, T_2) \cup \{ 0 \}$ are one dimensional.  By Corollary \ref{D}, we may extend the operators $T_1, T_2$ to $N$ densely d-hypercylic operators $T_1, \dots, T_N$.  Furthermore, 
    since $\mbox{d-}\mathcal{HC}(T_1, \dots, T_N) \subseteq \mbox{d-}\mathcal{HC}(T_1, T_2)$, the only subspaces contained within the set d-$\mathcal{HC}(T_1,\dots, T_N) \cup \{ 0 \}$ are one dimensional.  Clearly, in this situation, these densely d-hypercyclic operators fail to possess a dense d-hypercyclic manifold.
\end{proof}

Formally, operators $T_1, \dots, T_N \in \mathcal{B}(X)$ with $N \ge 2$ satisfy the \textit{Strong Disjoint Blow-up/Collapse property} (respectively, \textit{Disjoint Blow-up/Collapse property}) provided for each $K \in \mathbb{N}$ (respectively, for $K = 1$ only) and for any non-empty open sets $W$, $U_{1-K}, \dots, U_0, U_1, \dots, U_N$ of $X$ with $0 \in W$, there exists an integer $n \in \mathbb{N}$ such that the following holds:
\begin{align}
    W \cap T_{1}^{-n}(U_1) \cap \cdots \cap T_N^{-n}(U_N) 
    \ne \emptyset
    \mbox{ and }
    U_{\ell} \cap T_{1}^{-n}(W) \cap \cdots \cap T_N^{-n}(W) 
    \ne \emptyset
    \mbox{ for integers $1-K \le \ell \le 0$.}
    \nonumber
\end{align}
In the disjoint hypercyclicity literature, the Disjoint Blow-up/Collapse property arose first as the disjoint analog of the Blow-up/Collapse property.  In \cite{Sa4}, Salas provided the stronger version of the Disjoint Blow-up/Collapse property and showed this Strong Disjoint Blow-up/Collapse property is a sufficient condition to ensure the existence of a dense d-hypercyclic manifold for a finite family of operators. In the same paper, Salas also showed that among the class of weighted shift operators Disjoint Blow-up/Collapse property and its stronger version are equivalent. In the next corollary, we show that this is not the case in general and, therefore, answer Question \ref{Q2} in the negative.

\begin{corollary}\label{C}
    Every separable, infinite dimensional Banach space $X$ admits operators $T_1,T_2$ which satisfy the Disjoint Blow-up/Collapse property but fail to satisfy the Strong Disjoint Blow-up/Collapse property.  
\end{corollary}

\begin{proof}
    Let $T_1$ be any operator in $\mathcal{B}(X)$ which satisfies the Hypercyclicity Criterion, and let $Y = \mbox{span} \{ f \}$ where $f$ is a hypercyclic vector of $T_1$.   Applying Theorem \ref{d-subspace} with $M = 2$ yields an operator $T_2$ such that the the direct sum operators $T_1 \oplus T_1$, $T_2 \oplus T_2$ are densely d-hypercyclic, but the set d-$\mathcal{HC}(T_1,T_2) \cup \{ 0 \}$ fails to contain a subspace of dimension great than two.  Hence, the densely d-hypercyclic operators $T_1,T_2$ fail to possess a dense d-hypercyclic manifold.  Since the direct sum operators $T_1 \oplus T_1$, $T_2 \oplus T_2$ are densely d-hypercyclic, the operators $T_1,T_2$ must satisfy the Disjoint Blow-up/Collapse property (see \cite[Proposition 1.11]{BeMaSa}).  On the other hand, Salas \cite{Sa4} showed that operators which satisfy the Strong Disjoint Blow-up/Collapse property must possess dense d-hypercyclic manifold, and so these densely d-hypercyclic operators $T_1,T_2$ fail to satisfy the Strong Disjoint Blow-up/Collapse property.
\end{proof}

Even though the Disjoint Blow-up/Collapse property and the Strong Disjoint Blow-up/Collapse property are similar in presentation, there is a significant difference between the operators satisfying these two properties.  From the proof of our Corollary \ref{C}, we see that operators satisfying the Disjoint Blow-up/Collapse property may fail to have a dense d-hypercyclic manifold while those operators satisfying the Strong version of the property must have such a dense manifold.




\end{document}